\documentclass[review]{elsarticle}

\usepackage{bm,amssymb,amsmath,mathrsfs}
\usepackage{amsthm}
\usepackage{lineno}
\usepackage[usenames]{color}

\usepackage{dcolumn}

\theoremstyle{plain}
\newtheorem{theorem}{Theorem}[section]
\newtheorem*{theorem*}{Theorem}
\newtheorem{proposition}[theorem]{Proposition}

\begin{document}

\journal{(internal report CC24-11)}

\begin{frontmatter}

\title{Solutions of inhomogeneous linear difference equations using Green's functions}

\author[cc]{S.~R.~Mane}
\ead{srmane001@gmail.com}
\address[cc]{Convergent Computing Inc., P.~O.~Box 561, Shoreham, NY 11786, USA}

\begin{abstract}
We present a general formula for the particular solution of an inhomogeneous linear difference equation with variable coefficients.
The answer is expressed as a weighted sum of fundamental solutions of the associated linear difference equation.
This corresponds to an initial value problem in the case of linear differential equations.
We remark that Green's functions are naturally suited for solving such problems.
This note presents a Green's function formalism to solve an inhomogeneous linear difference equation with variable coefficients.
Both the retarded and advanced Green's functions are required, to obtain a complete solution.
We independently confirm previous work for the case of linear difference equations with constant coefficients.
\end{abstract}


\end{frontmatter}

\newpage
\setcounter{equation}{0}
\section{Introduction}\label{sec:intro}
It is well-known that the formalism of Green's functions furnishes a general method to derive a particular solution of an inhomogeneous linear differential equation \cite{JJ}.
Symbolically, let the inhomogeneous linear differential equation be $L[f(x)] = r(x)$, where $L$ is a linear differential operator, and $r(x)$ does not depend on $f(x)$.
The Green's function $G(x,x^\prime)$ is defined to satisfy the equation $L[G(x,x^\prime)] = \delta(x-x^\prime)$.
A particular solution $P(x)$ of the original equation is given by $P(x) = \int G(x,x^\prime) r(x^\prime)\,dx^\prime$.
Suitable boundary conditions must be specified to make $G(x,x^\prime)$ well defined, e.g.~a retarded or advanced Green's function (also the end-points for the above integral).
It is also known that if the linear differential operator $L[\cdot]$ has constant coefficients, then $G(x,x^\prime)$ is a function of only the difference $x-x^\prime$.

It is also well-known that there are close connections between the theories of linear differential equations and linear difference equations.
For example, a characteristic polynomial is employed to solve both linear differential equations with constant coefficients and linear difference equations with constant coefficients.
Less is known about linear difference equations with variable coefficients.
There are many papers on techniques to derive the solutions of such equations.
The textbook by Kelley and Peterson \cite{Kelley_Petersen} presents an exposition on linear difference equations.
The papers by Birkhoff \cite{Birkhoff} and Trjitzinsky \cite{Trjitzinsky} are noteworthy.
See also the papers by
Harris and Sibuya \cite{Harris_Sibuya},
Van der Cruyssen \cite{Cruyssen} and
Wimp and Zeilberger \cite{WimpZeilberger}.
Culmer and Harris \cite{Culmer_Harris} formulated the problem in matrix form, as a set of coupled first-order difference equations.
A notable more recent reference was published by Mallik \cite{Mallik},
who derived explicit solutions as nested sums of combinatorial expressions over the coefficients of the difference equation.
Mallik treated both linear difference equations of finite order $N\in\mathbb{Z}_+$ and of unbounded order.
See also the solution using a companion matrix for the case of difference equations of finite order $N$ (\cite{Mallik}, eq.~(16)).

Our goal in this note is more modest.
We offer a general formula for the particular solution of an inhomogeneous linear difference equation with variable coefficients.
We employ a Green's function formalism to do so, borrowing ideas from the theory of linear differential equations.
Wolfram \cite{Wolfram2000} published a general formula for the particular solution of an inhomogeneous linear difference equation with constant coefficients.
The case of constant coefficients is an important special case, and it will be shown below that in this case our expression matches Wolfram's solution.

\setcounter{equation}{0}
\section{Linear difference equation}\label{sec:lineq}
Consider the following inhomogeneous linear difference equation of order $d\ge2$
\begin{equation}
\label{eq:diff_eq_inhom}
c_0(n)f(n) +c_1(n)f(n-1) +\dots +c_d(n)f(n-d) = r(n) \,.
\end{equation}
Here $n$ is an integer, and $f$, $r$ and $c_i$, $i=0,\dots,d$ are real or complex valued functions of $n$.
Note that we allow the value of $n$ to be negative.
(For example, the combinatorial formulas by Mallik \cite{Mallik} for $f(n)$ require $n$ to be nonnegative.)
\emph{Linearity:} both $c_i(n)$ and $r(n)$ can depend on $n$, but they cannot depend on $f$.
\emph{Note:} To avoid complications and needless special cases below, we demand that the first and last coefficients be nonzero for all $n$,
i.e.~$c_0(n)\ne0$ and $c_d(n)\ne0$ for all $n\in\mathbb{Z}$.
The corresponding homogeneous linear difference equation is
\begin{equation}
\label{eq:diff_eq_hom}
c_0(n)f(n) +c_1(n)f(n-1) +\dots +c_d(n)f(n-d) = 0 \,.
\end{equation}
\emph{Fundamental solutions:} It is known that any set of $d$ linearly independent solutions of eq.~\eqref{eq:diff_eq_hom},
say $\{F_i(n),\,i=0,\dots,d-1\}$ constitute a set of fundamental solutions, by which is meant that any solution of eq.~\eqref{eq:diff_eq_hom}
can be expressed as a linear combination of the $F_i(n)$.
\emph{Canonical basis solutions:} 
Let us denote the set of `canonical basis' solutions of eq.~\eqref{eq:diff_eq_hom} by $\{B_i(n),\,i=0,\dots,d-1\}$,
where the initial values are $B_i(n) = \delta_{i,n}$ for $n=0,\dots,d-1$.
The canonical basis solutions are an important set of fundamental solutions, but not the only possibility.
Next, the complementary solution of eq.~\eqref{eq:diff_eq_inhom} is any solution of the homogeneous equation eq.~\eqref{eq:diff_eq_hom}.
Given a set of initial values $f(n)=\alpha_n$ for $n=0,\dots,d-1$ for eq.~\eqref{eq:diff_eq_inhom},
we shall specify the complementary solution $C(n)$ to be
\begin{equation}
\label{eq:cs}
C(n) = \sum_{i=0}^{d-1} \alpha_iB_i(n) \,.
\end{equation}
Hence $C(n)=\alpha_n$ for $n=0,\dots,d-1$.
This choice for $C(n)$ implies that the particular solution of eq.~\eqref{eq:diff_eq_inhom}, say $P(n)$, vanishes for $n=0,\dots,d-1$.
This policy determines the expression for the particular solution uniquely.
The full solution of eq.~\eqref{eq:diff_eq_inhom} is the sum of the complementary and particular solutions, i.e.
\begin{equation}
f(n) = C(n) + P(n) \,.
\end{equation}
We employ a Green's function formalism to obtain the particular solution of eq.~\eqref{eq:diff_eq_inhom}.
Since we constrain $P(n)=0$ for $n=0,\dots,d-1$, we must solve for $P(n)$ in two nonoverlapping zones (i) $n\ge d$ and (ii) $n<0$.  

\subsection{Retarded Green's function}
For $n\ge d$, we work upwards from $n=d$ to $n\to\infty$.
Hence we employ the retarded Green's function, say $G_r(n,m)$.
Here $G_r(n,m)$ is the solution of the following inhomogeneous linear difference equation
\begin{equation}
\label{eq:diff_eq_Gret}
c_0(n)G_r(n,m) +c_1(n)G_r(n-1,m) +\dots +c_d(n)G_r(n-d,m) = \delta_{m,n} \,.
\end{equation}
The conditions on $G_r(n,m)$ are as follows.
\begin{enumerate}
\item
  If $n\in[m-d+1,m-1]$, then $G_r(n,m)=0$.
\item
  If $n=m$, then $G_r(n,m)=1/c_0(m)$. (Recall we demand $c_0(m)\ne0$.)
\item
  If $n>m$, then $G_r(n,m)$ satisfies the homogenous difference equation eq.~\eqref{eq:diff_eq_hom}.
\end{enumerate}
The reason for the first condition above is that the recurrence is of order $d$,
hence we can only demand that $G_r(n,m)$ vanish in the interval of $d-1$ consecutive values
from $m-1$ downwards.
The second condition is to match the Green's function to the Kronecker $\delta$.
The first and second conditions together supply a set of $d$ consecutive boundary conditions, which make the expression for $G_r(n,m)$ unique.
The third condition implies $G_r(n,m)$ is a linear combination of the fundamental solutions $F_i(n)$, $i=0,\dots,d-1$.
The retarded Green's function is a ratio of two determinants. First define the determinant
\begin{equation}
\label{eq:det_r}  
\Delta_r(n) = \left|\begin{array}{llll}
F_0(m-d+1) & \dots & F_{d-1}(m-d+1) \\
\vdots & \dots & \vdots \\
F_0(m-1) & \dots & F_{d-1}(m-1) \\
F_0(n) & \dots & F_{d-1}(n) \end{array}\right|\,.
\end{equation}
Then
\begin{equation}
\label{eq:Gret_sol}
G_r(n,m) = \frac{1}{c_0(m)}\frac{\Delta_r(n)}{\Delta_r(m)} \,.
\end{equation}
Observe that this linear combination satisfies the constraints for $n\in[m-d+1,m]$.
Then the particular solution of eq.~\eqref{eq:diff_eq_inhom}, say $P_r(n)$, is given by the sum
\begin{equation}
\label{eq:ps_ret}
 P_r(n) = \sum_{m=d}^n G_r(n,m)r(m) \qquad (n\ge d)\,.
\end{equation}
The sum on the right-hand side is the analog of an integral in the theory of linear differential equations.
The upper limit is $m=n$ because $G_r(n,m)=0$ if $n<m$.
The lower limit is $m=d$ because we constrain $P_r(n)=0$ for $n=0,\dots,d-1$.

\subsection{Advanced Green's function}
For $n<0$, we work downwards from $n=-1$ to $n\to-\infty$.
Hence we employ the advanced Green's function, say $G_a(n,m)$.
First, we shift $n\gets n+d$ in eq.~\eqref{eq:diff_eq_Gret} to obtain
\begin{equation}
c_0(n+d)G_a(n+d,m) +c_1(n+d)G_a(n+d-1,m) +\dots +c_d(n+d)G_a(n,m) = \delta_{m,n} \,.
\end{equation}
This is because we work downwards in $n$, so we want $G_a(n,m)$ in the \emph{last} term.
The conditions on $G_a(n,m)$ are as follows (note that $m$ is negative).
\begin{enumerate}
\item
  If $n\in[m+1,m+d-1]$, then $G_a(n,m)=0$.
\item
  If $n=m$, then $G_a(n,m)=1/c_d(m+d)$. (Recall we demand $c_d(m+d)\ne0$.)
\item
  If $n<m$, then $G_a(n,m)$ satisfies the homogenous difference equation eq.~\eqref{eq:diff_eq_hom}.
\end{enumerate}
The reason for the first condition above is that the recurrence is of order $d$,
hence we can only demand that $G_r(n,m)$ vanish in the interval of $d-1$ consecutive values from $m+1$ upwards.
The second condition is to match the Green's function to the Kronecker $\delta$.
The first and second conditions together supply a set of $d$ consecutive boundary conditions, which make the expression for $G_a(n,m)$ unique.
The third condition implies $G_a(n,m)$ is a linear combination of the fundamental solutions $F_i(n)$, $i=0,\dots,d-1$.
The advanced Green's function is a ratio of two determinants. First define the determinant
\begin{equation}
\label{eq:det_a}  
\Delta_a(n) = \left|\begin{array}{llll}
F_0(m+d-1) & \dots & F_{d-1}(m+d-1) \\
\vdots & \dots & \vdots \\
F_0(m+1) & \dots & F_{d-1}(m+1) \\
F_0(n) & \dots & F_{d-1}(n) \end{array}\right|\,.
\end{equation}
Then
\begin{equation}
\label{eq:Gadv_sol}
G_a(n,m) = \frac{1}{c_d(m+d)}\frac{\Delta_a(n)}{\Delta_a(m)} \,.
\end{equation}
Observe that this linear combination satisfies the constraints for $n\in[m,m+d-1]$.
Then the particular solution of eq.~\eqref{eq:diff_eq_inhom}, say $P_a(n)$, is given by the sum
\begin{equation}
\label{eq:ps_adv}
P_a(n) = \sum_{m=n}^{-1} G_a(n,m)r(m+d) \,.
\end{equation}
The sum on the right-hand side is the analog of an integral in the theory of linear differential equations.
The lower limit is $m=n$ because $G_a(n,m)=0$ if $n>m$.
The upper limit is $m=-1$ because we constrain $P_a(n)=0$ for $n=0,\dots,d-1$.

\subsection{Full solution}
From the foregoing, the full solution of eq.~\eqref{eq:diff_eq_inhom} is
\begin{equation}
\label{eq:full_sol}
f(n) = \begin{cases} C(n) & \qquad (0 \le n \le d-1) \,, \\
  \displaystyle C(n) +\sum_{m=d}^n G_r(n,m)r(m) & \qquad (n \ge d) \,, \\
  \displaystyle C(n) +\sum_{m=n}^{-1} G_a(n,m)r(m+d) & \qquad (n < 0) \,.
\end{cases}
\end{equation}

\subsection{Comment on determinants}
It is not necessary to specify any particular choice for the fundamental solutions $F_i(n)$ to construct the determinants
$\Delta_r(n)$ and $\Delta_a(n)$ in eqs.~\eqref{eq:det_r} and \eqref{eq:det_a} above.
The use of the canonical basis functions $B_i(n)$ is useful but not necessary.
Since every solution of the homogeneous equation eq.~\eqref{eq:diff_eq_hom} is a linear combination of the $F_i(n)$,
the \emph{ratio} of determinants in eqs.~\eqref{eq:Gret_sol} and \eqref{eq:Gadv_sol} remains unchanged
if we employ \emph{any} set of $d$ linearly independent solutions of eq.~\eqref{eq:diff_eq_hom}.
This fact will be helpful when constructing example solutions below.

\setcounter{equation}{0}
\section{Wronskian \&\ Casoratian}\label{sec:W_C}
The denominator determinants in eqs.~\eqref{eq:Gret_sol} and \eqref{eq:Gadv_sol} are Casoratians (discrete analogs of Wronskians).
There is a close connection of the determinants $\Delta_r(n)$ and $\Delta_a(n)$ for linear difference equations and Wronskians for linear ordinary differential equations.
Consider the following inhomogeneous linear ordinary differential equation 
\begin{equation}
\label{eq:ode}
\gamma_0(x)y^{(d)}(x) +\gamma_1(x)y^{(d-1)}(x) +\dots +\gamma_{d-1}(x)y^{(1)}(x) +\gamma_d(x)y(x) = \rho(x) \,.
\end{equation}
Here $x\in\mathbb{R}$ is the independent variable, $y(x)$ is the dependent variable and $y^{(i)}(x) = d^iy/dx^i$ denotes the $i^{th}$ derivative with respect to $x$.
Also $\gamma_i(x)$, $i=0,\dots,d$ and $\rho(x)$ are functions which depend on $x$ but not $y$.
To avoid complications, we assume all the $\gamma_i(x)$ and $\rho(x)$ are continuous functions of $x$ and $\gamma_0(x)\ne0$.
The corresponding homogeneous differential equation is
\begin{equation}
\label{eq:ode_hom}
\gamma_0(x)y^{(d)}(x) +\gamma_1(x)y^{(d-1)}(x) +\dots +\gamma_{d-1}(x)y^{(1)}(x) +\gamma_d(x)y(x) = 0 \,.
\end{equation}
Let $\{Y_i(x),\, i=0,\dots,d-1\}$ be a set of $d$ linearly independent solutions, i.e.~a set of fundamental solutions, of eq.~\eqref{eq:ode_hom}.
The differential equation for the Green's function $G(x,\xi)$ is
\begin{equation}
\label{eq:ode_G}
\gamma_0(x)G^{(d)}(x) +\gamma_1(x)G^{(d-1)}(x) +\dots +\gamma_{d-1}(x)G^{(1)}(x) +\gamma_d(x)G(x) = \delta(x-\xi) \,.
\end{equation}
Define the Wronskian at $x=\xi$ as follows
\begin{equation}
\label{eq:W_xi}
W(\xi) = \left|\begin{array}{lllll} Y_0(\xi) & \dots & Y_{d-1}(\xi) \\ \vdots & \dots & \vdots \\
Y_0^{(d-2)}(\xi) & \dots & Y_{d-1}^{(d-2)}(\xi) \\
Y_0^{(d-1)}(\xi) & \dots & Y_{d-1}^{(d-1)}(\xi) \end{array}\right|\,.
\end{equation}
Also define the following numerator determinant, which is a function of both $x$ and $\xi$:
\begin{equation}
\label{eq:N_xi}
N(x,\xi) = \left|\begin{array}{lllll} Y_0(\xi) & \dots & Y_{d-1}(\xi) \\ \vdots & \dots & \vdots \\
Y_0^{(d-2)}(\xi) & \dots & Y_{d-1}^{(d-2)}(\xi) \\
Y_0(x) & \dots & Y_{d-1}(x) \end{array}\right|\,.
\end{equation}
For the retarded Green's function, say $G_r(x,\xi)$,
we demand that $G_r(x,\xi)=0$ for $x<\xi$ and $G_r(x,\xi)$ satisfies eq.~\eqref{eq:ode_hom} for $x>\xi$.
The retarded Green's function is given by
\begin{equation}
\label{eq:Gret_ode}
G_r(x,\xi) = \begin{cases} 0 & \qquad (x < \xi)\,, \\ \displaystyle \frac{1}{\gamma_0(\xi)}\frac{N(x,\xi)}{W(\xi)} & \qquad (x \ge \xi) \,. \end{cases}
\end{equation}
This is analogous to eq.~\eqref{eq:Gret_sol} for the difference equation.
For the advanced Green's function, say $G_a(x,\xi)$,
we demand that $G_a(x,\xi)=0$ for $x>\xi$ and $G_a(x,\xi)$ satisfies eq.~\eqref{eq:ode_hom} for $x<\xi$.
The advanced Green's function is given by
\begin{equation}
\label{eq:Gadv_ode}
G_a(x,\xi) = \begin{cases} \displaystyle -\frac{1}{\gamma_0(\xi)}\frac{N(x,\xi)}{W(\xi)} & \qquad (x \le \xi)\,, \\ 0 & \qquad (x > \xi)\,. \end{cases}
\end{equation}
This is slightly different from eq.~\eqref{eq:Gadv_sol} for the difference equation,
because for a differential equation, the function value $y(x)$ and all its derivatives are evaluated at the same value of $x$.
Hence the denominator is $\gamma_0(\xi)$ for both the retarded and advanced Green's functions of a linear differential equation.

\emph{Comment:}
Note that we said $x\in\mathbb{R}$ above.
We could also have said the independent variable is complex $z\in\mathbb{C}$,
but then we would have to specify a contour or path of integration in the complex plane, to define the concepts of `retarded' and `advanced' Green's functions.

\setcounter{equation}{0}
\section{Constant coefficients}\label{sec:constcoeff}
Wolfram \cite{Wolfram2000} published a general solution for linear difference equations with constant coefficients.
Wolfram's linear difference equation for a function $f(n)$ is (\cite{Wolfram2000}, eq.~(1.1))
\begin{equation}
\label{eq:W_rec}
f(n) -\sum_{\ell=1}^d a_{d-\ell} f(n-\ell) = r(n) \,.
\end{equation}
Here $a_i$, $i=0,\dots,d-1$ are a set of constant coefficients.
We must have $a_0\ne0$, else the recurrence would not be of order $d$. The other $a_i$ can be zero, for $i=1,\dots,d-1$.  
Also $r(n)$ is an inhomogeneous term, which can depend on $n$ but not on $f$.
Compare with eq.~\eqref{eq:diff_eq_inhom}, then $c_0(n)=1$ and $c_\ell(n) = -a_{d-\ell}$ for $\ell=1,\dots,d$. 
Observe that $c_0(n)\ne0$ and $c_d(n)=-a_0\ne0$.
The corresponding homogeneous linear difference equation is
\begin{equation}
\label{eq:W_rec_homog}
f(n) -\sum_{\ell=1}^d a_{d-\ell} f(n-\ell) = 0 \,.
\end{equation}
In this section, the `canonical basis' solutions $B_i(n)$, $i=0,\dots,d-1$ refer to eq.~\eqref{eq:W_rec_homog}.
There are standard techniques to determine the $B_i$ in this case, e.g.~by finding the roots of the characteristic equation of eq.~\eqref{eq:W_rec_homog}.
We do not discuss the matter here.
Suppose again the initial values are $\{\alpha_0,\dots,\alpha_{d-1}\}$ for $n=0,\dots,d-1$ and the complementary solution is $C(n)$
as in eq.~\eqref{eq:cs},
\begin{equation}
C(n) = \sum_{i=0}^{d-1} \alpha_iB_i(n) \,.
\end{equation}
This choice for $C(n)$ implies that the particular solution of eq.~\eqref{eq:W_rec} vanishes for $n=0,\dots,d-1$.
This policy determines the expression for the particular solution uniquely.
Wolfram published a general solution for $f(n)$, with the above constraints.
We express Wolfram's solution (\cite{Wolfram2000}, eq.~(2.2)) as follows.
\begin{equation}
\label{eq:W_my_sol}
f(n) = \begin{cases} C(n) & \qquad (0 \le n \le d-1) \,, \\
  \displaystyle C(n) +\sum_{i=0}^{n-d} r(n-i)B_{d-1}(d-1+i) & \qquad (n \ge d) \,, \\
  \displaystyle C(n) -\sum_{i=1}^{-n} r(n+d+i-1)B_{d-1}(-i) & \qquad (n < 0) \,.
\end{cases}
\end{equation}
We shall derive eq.~\eqref{eq:W_my_sol} using Green's functions.
\begin{proposition}
The retarded Green's function is given by
\begin{equation}
\label{eq:Gret_const_coeff}
G_r(n,m) = B_{d-1}(d-1+n-m) \,.
\end{equation}
\end{proposition}
\begin{proof}
The right-hand side of eq.~\eqref{eq:Gret_const_coeff}
equals (i) zero for $n\in[m-d,m-1]$, (ii) $1$ for $n=m$ and (iii) satisfies the homogeneous recurrence eq.~\eqref{eq:diff_eq_hom} for $n>m$.
\end{proof}
Then the particular solution is, from eq.~\eqref{eq:ps_ret},
\begin{equation}
P_r(n) = \sum_{m=d}^n r(m)G_r(n,m) = \sum_{m=d}^n r(m)B_{d-1}(d-1+n-m) \,.
\end{equation}
Set $i=n-m$, then we obtain
\begin{equation}
P(n) = \sum_{i=0}^{n-d} r(n-i)B_{d-1}(d-1+i) \,.
\end{equation}
This equals the expression in the second line of eq.~\eqref{eq:W_my_sol}.
\begin{proposition}
The advanced Green's function is given by
\begin{equation}
\label{eq:Gadv_const_coeff}
G_a(n,m) = -B_{d-1}(n-m-1) \,.
\end{equation}
\end{proposition}
\begin{proof}
First set $n=d-1$ in eq.~\eqref{eq:diff_eq_hom} to deduce (recall $B_{d-1}(d-1)=1$ and $B_{d-1}(n)=0$ for $n=0,\dots,d-2$)
\begin{equation}
\begin{split}
  0 &= B_{d-1}(d-1) -a_{d-1}B_{d-1}(d-2) -\dots -a_0(n)B_{d-1}(-1) 
  \\
  &= 1 -a_0B_{d-1}(-1) \,.
\end{split}
\end{equation}
Hence $B_{d-1}(-1) = 1/a_0$.  
Then the right-hand side of eq.~\eqref{eq:Gadv_const_coeff}
is (i) zero for $0\ge n>m$, (ii) $-1/a_0$ for $n=m$ and (iii) satisfies the homogeneous recurrence eq.~\eqref{eq:diff_eq_hom} for $n<m$.
\end{proof}  
Then the particular solution is, from eq.~\eqref{eq:ps_adv},
\begin{equation}
P_a(n) = -\sum_{m=n}^{-1} r(m+d)B_{d-1}(n-m-1) \,.
\end{equation}
Set $i=1-n+m$, then we obtain
\begin{equation}
P_a(n) = -\sum_{i=1}^{-n} r(n+d+i-1)B_{d-1}(-i) \,.
\end{equation}
This equals the expression in the last line of eq.~\eqref{eq:W_my_sol}.

Let us derive the above expressions for the retarded and advanced Green's function using the formalism in Sec.~\ref{sec:lineq}.
First observe that if the coefficients are constant,
the homogeneous equation eq.~\eqref{eq:W_rec_homog} is translation invariant,
in the sense that $g(n) = f(n+j)$ also satisfies eq.~\eqref{eq:W_rec_homog}, for all $j\in\mathbb{Z}$.
Then the expressions for the Green's functions simplify.
For the retarded Green's function, instead of treating $\delta_{m,n}$, shift all the indices by $-m+d-1$ and treat $\delta_{d-1,n}$.
Also employ the canonical basis functions.
Observe that
\begin{equation}
\Delta_r(n-m+d-1) = \left|\begin{array}{llll}
B_0(0) & \dots & B_{d-1}(0) \\
\vdots & \dots & \vdots \\
B_0(d-2) & \dots & B_{d-1}(d-2) \\
B_0(n-m+d-1) & \dots & B_{d-1}(n-m+d-1) \end{array}\right| \,.
\end{equation}
The only nonzero terms in the first $d-1$ rows are the diagonal terms, and they all equal unity.
\emph{This is why we require the canonical basis functions, to simplify the structure of determinant.}
Hence
\begin{equation}
\label{eq:Gret_sol_constcoeff}
G_r(n,m) = \frac{1}{c_0(n)}\frac{B_{d-1}(n-m+d-1)}{B_{d-1}(d-1)} = B_{d-1}(n-m+d-1) \,.
\end{equation}
This matches eq.~\eqref{eq:Gret_const_coeff}.
Next process the advanced Green's function.
By the same logic as above, shift all the indices by $-m$ and treat $\delta_{0,n}$.
Also employ the canonical basis functions, but arrange the columns as shown below.
Observe that
\begin{equation}
\Delta_a(n-m) = \left|\begin{array}{llll}
B_{d-1}(d-1) & \dots & B_0(d-1) \\
\vdots & \dots & \vdots \\
B_{d-1}(1) & \dots & B_0(1) \\
B_{d-1}(n-m) & \dots & B_0(n-m) \end{array}\right|\,.
\end{equation}
The only nonzero terms in the first $d-1$ rows are the diagonal terms, and they all equal unity.
\emph{This is why we require the canonical basis functions, and this time we order them with $B_{d-1}$ in the first column, etc.}
Hence
\begin{equation}
G_a(n,m) = \frac{1}{c_d(n)}\frac{B_0(n-m)}{B_0(0)}= -\frac{B_0(n-m)}{a_0} \,.
\end{equation}
For the case of constant coefficients, Wolfram proved that $B_0(n)=a_0B_{d-1}(n-1)$ (\cite{Wolfram2000}, Corollary 2.3). Hence
\begin{equation}
\label{eq:Gadv_sol_constcoeff}
G_a(n,m) = -B_{d-1}(n-m-1) \,.
\end{equation}
This matches eq.~\eqref{eq:Gadv_const_coeff}.

\setcounter{equation}{0}
\section{Example \#1}\label{sec:ex1}
Let us consider a simple example to demonstrate.
It is tempting to employ the Fibonacci recurrence, but let us use something different.
Consider the following recurrence, with an inhomogeneous term $r(n)=n$.
\begin{equation}
\label{eq:ex1}  
f(n) = 2f(n-1) -f(n-2) +n \,.
\end{equation}
The characteristic equation is $z^2-2z+1=(z-1)^2=0$, i.e.~a repeated root.
The solution of the homogeneous equation is $f(n) = u + vn$, where $u$ and $v$ are arbitrary constants.
The canonical basis solutions are $B_0(n)=1-n$ and $B_1(n)=n$.
We are interested only in the `forced' or `driven' solution generated by the inhomogeneous term $r(n)=n$.
Hence we ignore the complementary solution $C(n)$ below and solve for the particular solution $P(n)$, subject to $P(0)=P(1)=0$.
Note that the inhomogenous term $r(n)=n$ coincides with $B_1(n)$.
Hence there is a `resonance' of the driving term with a solution of the inhomogeneous recurrence.
It is easily verified that the solution is of $O(n^3)$.
\begin{equation}
\label{eq:sol_ex1}
P(n) = \frac16 n(n-1)(n+4) \,.
\end{equation}
Observe that $P(0)=P(1)=0$. Direct substitution into eq.~\eqref{eq:ex1} yields
\begin{equation}
P(n) -2P(n-1) +P(n-2) = \frac16n(n-1)(n+4) -\frac13(n-1)(n-2)(n+3) +\frac16(n-2)(n-3)(n+2) = n \,.
\end{equation}
Let us employ the retarded Green's function for $n\ge2$, using the second line of eq.~\eqref{eq:W_my_sol}, with $C(n)=0$.
\begin{equation}
\begin{split}
  f(n) &= \sum_{i=0}^{n-2} r(n-i)B_1(1+i) \\
  &= \sum_{i=0}^{n-2} (n-i)(i+1) \\
  &= \frac12n^2(n-1) -\frac13 n(n-1)(n-2) \\
  &= \frac16n(n-1)(n+4) \,.
\end{split}
\end{equation}
This matches eq.~\eqref{eq:sol_ex1}.
Next let us employ the advanced Green's function for $n<0$, using the last line of eq.~\eqref{eq:W_my_sol}, with $C(n)=0$.
\begin{equation}
\begin{split}
  f(n) &= -\sum_{i=1}^{-n} r(n+1+i)B_1(-i) \\
  &= \sum_{i=1}^{-n} (n+1+i)i \\
  &= \frac12 n(-n)(-n+1) +\frac13(-n)(-n+1)(-n+2) \\
  &= \frac12 n^2(n-1) -\frac13 n(n-1)(n-2) \\
  &= \frac16n(n-1)(n+4) \,.
\end{split}
\end{equation}
This matches eq.~\eqref{eq:sol_ex1}.

\setcounter{equation}{0}
\section{Example \#2}\label{sec:ex2}
We treat the same homogeneous recurrence in Sec.~\ref{sec:ex1}, but with a nonresonant driving term $r(n)=2^n$.
\begin{equation}
\label{eq:ex2}  
f(n) = 2f(n-1) -f(n-2) +2^n \,.
\end{equation}
The canonical basis solutions are the same: $B_0(n)=1-n$ and $B_1(n)=n$.
We are again interested only in the `forced' or `driven' solution generated by the inhomogeneous term $r(n)=1/n$.
Hence we again ignore the complementary solution $C(n)$ below and solve for the particular solution $P(n)$, subject to $P(0)=P(1)=0$.
The solution for the driven solution is
\begin{equation}
\label{eq:sol_ex2}
P(n) = 2^{n+2} -4B_0(n) -8B_1(n) = 2^{n+2} -4 -4n \,.
\end{equation}
Observe that $P(0)=0$ and $P(1)=0$, as required.
The functions $B_0(n)$ and $B_1(n)$ do not contribute to the inhomogenous equation.
Substitute $f(n)=2^{n+2}$ and observe that
\begin{equation}
f(n)-2f(n-1)+f(n-2) = 2^{n+2} -2\times2^{n+1} +2^n = 2^n \,.
\end{equation}
Next let us employ the retarded Green's function for $n\ge 2$, using the second line of eq.~\eqref{eq:W_my_sol}, with $d=2$ and $C(n)=0$.
\begin{equation}
\begin{split}
  f(n) &= \sum_{i=0}^{n-2} r(n-i)B_1(1+i) \\
  &= \sum_{i=0}^{n-2} 2^{n-i}(i+1) \\
  &= 2^n\Bigl(2 -\frac{n}{2^{n-2}}\Bigr) +2^{n+1}-4 \\
  &= 2^{n+2} -4n -4 \,.
\end{split}
\end{equation}
This matches eq.~\eqref{eq:sol_ex2}.
Next let us employ the advanced Green's function for $n<0$, using the last line of eq.~\eqref{eq:W_my_sol}, with $d=2$ and $C(n)=0$.
\begin{equation}
\begin{split}
  f(n) &= -\sum_{i=1}^{|n|} r(n+1+i)B_d(-i) \\
  &= \sum_{i=1}^{|n|} 2^{n+1+i}i \\
  &= 2^{n+2}(2^{|n|} |n|-2^{|n|}+1) \\
  &= 2^{n+2}(-2^{-n} n -2^{-n}+1) \\
  &= 2^{n+2} -4n -4 \,.
\end{split}
\end{equation}
This matches eq.~\eqref{eq:sol_ex2}.

\setcounter{equation}{0}
\section{Example \#3}\label{sec:ex3}
Next let us consider a linear difference equation with variable coefficients.
\begin{equation}
\label{eq:ex3}
  (2n-1)f(n) -4nf(n-1) +(2n+1)f(n-2) = 3 \,.
\end{equation}
It is easily verified that two linearly independent solutions of the homogeneous equation are $F_0(n)=1$ and $F_1(n)=(n+1)^2$.
They are \emph{not} canonical basis solutions.
Observe that the driving term $r(n)=3$ is resonant with $F_0(n)$.
We again solve only for the driven solution $P(n)$, with the constraints $P(0)=P(1)=0$.
The answer is
\begin{equation}
P(n) = \frac{n(n-1)}{2} \,.
\end{equation}
This can be verified by direct substitution.
Say the left-hand side of eq.~\eqref{eq:ex3} is $L$, then
\begin{equation}
\begin{split}
  L &= (2n-1)P(n) -4nP(n-1) +(2n+1)P(n-2) \\
  &= (2n-1)\frac{n(n-1)}{2} -2n(n-1)(n-2) +(2n+1)\frac{(n-2)(n-3)}{2} \\
  &= \frac12\bigl[ (2n-1)n(n-1) -4n(n-1)(n-2) +(2n+1)(n-2)(n-3) \bigr] \\
  &= 3 \,.
\end{split}
\end{equation}
Let us derive the particular solution using Green's functions.
From eq.~\eqref{eq:Gret_sol}, for $n\ge2$ the retarded Green's function is
\begin{equation}
\begin{split}
G_r(n,m) &= \frac{1}{c_0(m)}\frac{F_0(m-1)F_1(n)-F_1(m-1)F_0(n)}{F_0(m-1)F_1(m)-F_1(m-1)F_0(m)} \\
&= \frac{1}{2m-1}\frac{(n+1)^2-m^2}{(m+1)^2-m^2} \\
&= \frac{(n+1)^2-m^2}{(2m-1)(2m+1)} \,.
\end{split}
\end{equation}
Then sum to obtain the particular solution
\begin{equation}
\begin{split}
  P_r(n) &= \sum_{m=d}^n G_r(n,m)r(m)
  \\
  &= 3\sum_{m=2}^n \frac{(n+1)^2-m^2}{(2m-1)(2m+1)} \,.
\end{split}
\end{equation}
Process the sums separately. The first sum is
\begin{equation}
\begin{split}
  2S_1 &= \sum_{m=2}^n \frac{2}{(2m-1)(2m+1)} \\
  &= \sum_{m=2}^n \biggl(\frac{1}{2m-1} -\frac{1}{2m+1}\biggr) \\
  &= \frac13 - \frac{1}{2n+1} \\
  &= \frac23\frac{n-1}{2n+1} \,.
\end{split}
\end{equation}
The second sum is
\begin{equation}
\begin{split}
  4S_2 &= \sum_{m=2}^n \frac{4m^2}{4m^2-1} \\
  &= \sum_{m=2}^n \Bigl(1 +\frac{1}{4m^2-1}\Bigr) \\
  &= n-1 +S_1 \\
  &= n-1 +\frac13\frac{n-1}{2n+1} \\
  &= \frac23\frac{(3n+2)(n-1)}{2n+1} \,.
\end{split}
\end{equation}
The total is
\begin{equation}
\begin{split}
  P_r(n) &= (n+1)^2\frac{n-1}{2n+1} - \frac12\frac{(3n+2)(n-1)}{2n+1} \\
  &= \frac{n-1}{2}\frac{2(n+1)^2-3n-2}{2n+1} \\
  &= \frac{n(n-1)}{2} \,.
\end{split}
\end{equation}
This is correct.
Next process the case $n\le-1$.
From eq.~\eqref{eq:Gadv_sol}, the advanced Green's function is
\begin{equation}
\begin{split}
G_a(n,m) &= \frac{1}{c_2(m+2)}\frac{F_0(m+1)F_1(n)-F_1(m+1)F_0(n)}{F_0(m+1)F_1(m)-F_1(m+1)F_0(m)} \\
&= \frac{1}{2m+5}\frac{(n+1)^2-(m+2)^2}{(m+1)^2-(m+2)^2} \\
&= -\frac{(n+1)^2-(m+2)^2}{(2m+3)(2m+5)} \,.
\end{split}
\end{equation}
Then sum to obtain the particular solution
\begin{equation}
\begin{split}
  P_a(n) &= \sum_{m=n}^{-1} G_r(n,m)r(m)
  \\
  &= -3\sum_{m=n}^{-1} \frac{(n+1)^2-(m+2)^2}{(2m+3)(2m+5)}
  \\
  &= -3\sum_{j=1}^{|n|} \frac{(|n|-1)^2-(j-2)^2}{(2j-3)(2j-5)} \,.
\end{split}
\end{equation}
Process the sums separately. The first sum is
\begin{equation}
\begin{split}
  2S_3 &= \sum_{j=1}^{|n|} \frac{2}{(2j-3)(2j-5)} \\
  &= \sum_{j=1}^{|n|} \biggl(\frac{1}{2j-5} -\frac{1}{2j-3}\biggr) \\
  &= -\Bigl(\frac13 + \frac{1}{2|n|-3}\Bigr) \\
  &= -\frac23\frac{|n|}{2|n|-3} \,.
\end{split}
\end{equation}
The second sum is
\begin{equation}
\begin{split}
  4S_4 &= \sum_{j=1}^{|n|} \frac{4(j-2)^2}{(2j-3)(2j-5)} \\
  &= \sum_{j=1}^{|n|} \frac{4j^2-16j+16}{4j^2-16j+15} \\
  &= \sum_{j=1}^{|n|} \Bigl(1 +\frac{1}{4j^2-16j+15}\Bigr) \\
  &= |n| +S_3 \\
  &= |n| -\frac13\frac{|n|}{2|n|-3} \\
  &= \frac23\frac{3|n|^2-5|n|}{2|n|-3} \,.
\end{split}
\end{equation}
The total is
\begin{equation}
\begin{split}
  P_r(n) &= (|n|-1)^2\frac{|n|}{2|n|-3} + \frac12\frac{3|n|^2-5|n|}{2|n|-3} \\
  &= \frac{|n|}{2}\frac{2(|n|-1)^2 +3|n|-5}{2|n|-3} \\
  &= \frac12|n|(|n|+1) \\
  &= \frac{n(n-1)}{2} \,.
\end{split}
\end{equation}
This is correct.

\section{Conclusion}\label{sec:conc}
This note leverages the Green's function formalism from the theory of inhomogeneous linear differential equations
and applies it to inhomogeneous linear difference equations.
The advantage of a Green's function formalism is that it can be employed to solve a linear difference equation for \emph{any} inhomogeneous function $r(n)$.
It is only necessary to determine a set of fundamental solutions of the homogeneous equation once.
After that is done (which is admittedly a nontrivial task for linear difference equations with variable coefficients),
one can solve for an arbitrary inhomogeneous function $r(n)$.
It is not necessary to solve the difference equation \emph{ab initio} for each inhomogeneous function $r(n)$.
The Green's function formalism is a general-purpose technique for solving inhomogeneous linear difference equations.

\newpage
\section*{Acknowledgements}
The author is grateful to David Wolfram for bringing Ref.~\cite{Wolfram2000} to his attention and stimulating his interest in this problem.

\bibliographystyle{amsplain}

\begin{thebibliography}{10}
\bibitem{JJ}
H.~Jeffreys and B.~Jeffreys,
\textit{Methods of Mathematical Physics}, 3rd Edition,
Cambridge University Press, Cambridge, UK (1972).

\bibitem{Kelley_Petersen}
W.~G.~Kelley and A.~C.~Peterson,
\textit{Difference Equations: An Introduction with Applications},
Academic Press, San Diego, USA (1991).

\bibitem{Birkhoff}
G.~D.~Birkhoff,
\textit{Formal theory of irregular linear difference equations},
Acta Math.~\textbf{54} (1930), 205--246.

\bibitem{Trjitzinsky}
W.~J.~Trjitzinsky,
\textit{Laplace integrals and factorial series in the theory of linear differential and linear difference equations},
Trans.~Amer.~Math.~Soc.~\textbf{37} (1935), 80--146.

\bibitem{Harris_Sibuya}
W.~A.~Harris, Jr., and Y.~Sibuya,
\textit{Note on linear difference equations},
Bull.~Amer.~Math.~Soc.~\textbf{70} (1964), 123--127.
  
\bibitem{Cruyssen}
P.~Van der Cruyssen,
\textit{Linear Difference Equations and Generalized Continued Fractions},
Computing \textbf{22} (1979), 269--278.

\bibitem{WimpZeilberger}
J.~Wimp and D.~Zeilberger,
\textit{Resurrecting the Asymptotics of Linear Recurrences},\\
J.~Math.~Anal.~Appl.~\textbf{111} (1985), 162--176.

\bibitem{Culmer_Harris}
W.~J.~A. Culmer and W.~A.~Harris, Jr.,
\textit{Convergent solutions of ordinary linear homogeneous difference equations},
Pacific J.~Math.~\textbf{13} (1963), 1111--1138.
  
\bibitem{Mallik}
R.~K.~Mallik,
\textit{Solutions of Linear Difference Equations with Variable Coefficients},\\
J.~Math.~Anal.~Appl.~\textbf{222} (1998), 79--91.

\bibitem{Wolfram2000}
D.~A.~Wolfram,
\textit{A Formula for the General Solution of a Constant-coefficient Difference Equation},
J.~Symbolic Computation \textbf{29} (2000), 79--82.

\end{thebibliography}

\end{document}